\documentclass{amsart}
\usepackage{graphicx}
\usepackage {epsf}
\usepackage{slashed}
\usepackage{euscript}           
\usepackage{amsmath,amsthm}     
\usepackage{amssymb}            

\usepackage {epsf}

\newcommand {\Z}{{\mathbb Z}}

\newcommand{\A}{\mathcal A}

\newcommand{\ra}{\rightarrow}                   
\newcommand{\lra}{\longrightarrow}              
\newcommand{\colim}{\mbox{\rm co}\! \lim }

\newcommand{\hocolim}{\mbox{\rm hoco}\! \lim }

\newfont{\german}       {eufm10 at 12pt}
\DeclareMathOperator{\Hom}{Hom}

\newtheorem{thm}{Theorem}[section]
\newcounter{numerierer}
\newcounter{leer}

\newtheorem{defn}[thm]{Definition}
\newtheorem{prop}[thm]{Proposition}

\newtheorem{cor}[thm]{Corollary}
\newtheorem{lemma}[thm]{Lemma}

\theoremstyle{definition}  
\newenvironment{definition}{\begin{defn}\rm}{\end{defn}}

\newtheorem{example}[thm]{Example}

\usepackage[frame,matrix,curve, arrow]{xy}

\xymatrixcolsep{1.9pc}                          
\xymatrixrowsep{1.9pc}
\newdir{ >}{{}*!/-5pt/\dir{>}}                  
\newdir{ |}{{}*!/-5pt/\dir{|}}                  

\subjclass{}
\begin{document}

\title{Singularities and Quinn spectra}
\author{Nils A.\ Baas  and Gerd  Laures}

\address{
Department of Mathematical Sciences, NTNU,7491 Trondheim, Norway\newline\indent Fakult\"at f\"ur Mathematik,  Ruhr-Universit\"at Bochum, 44780 Bochum, Germany}
\date{\today }
\begin{abstract}
We introduce singularities to Quinn spectra. 
It enables us to talk about ads with prescribed singularities and to
explicitly construct
highly structured representatives for prominent spectra like
Morava $K$-theories or for $L$-theory with singularities. We develop a spectral sequence for the computation of the associated bordism groups and investigate product structures in the    
presence of singularities.\end{abstract}
\maketitle
\section{Introduction}
\label{sec:intro}
Manifolds with cone-like singularities were introduced by D. Sullivan in
\cite{S}. The concept was reformulated by Baas in \cite{MR0346824} as manifolds
with a higher order (multilevel) decomposition of its boundary. Based on	this definition a theory of cobordisms with singularities was developed. Many interesting homology and cohomology theories were constructed based	on this theory. For example the Morava $K$-theories, the Johnson--Wilson
theories, versions of elliptic cohomology, etc.

All these theories
have played an important role in homotopy theory and algebraic 
topology during the last 30--40 years.
However,
it is surprising how many results could be obtained by just knowing
their existence, not their construction. An explicit construction, however, can help in constructing important classes or in investigating the multiplicative structure of the representing spectra. Also, in order to obtain further results it
seems to be important that the spectra are related to the original
geometric category. 

This is the goal of the present paper. The theory of  `ads'  \cite{LM06} is used to construct Quinn-spectra with 
singularities. They are symmetric spectra which come with the expected long exact sequences and a Bousfield-Kan spectral sequence for the computation of their coefficients. Moreover, it turns out that these spectra always give strict module spectra over the original Quinn spectra. In some cases 
they even have an explicit $A_\infty$ structure or sometimes, as shown in \cite{LM14}  an $E_\infty$-structure.
If the Quinn spectrum is $L$-theory the singularities spectrum   seems to provide the natural surgery obstructions for manifolds with singularities.  

This work is organized as follows: we first recall from \cite{LM06} and \cite{LM14} the main results on ad theories and Quinn spectra and give a few examples. In Section 3 we introduce the singularities in the context of ads and develop new ad theories this way. Then the exact sequence for the bordism groups are constructed. It relates the ad theories among each other in case of a sequence of singularities.  Section 4 deals with the classical example of manifolds ads. An assembly map shows that the corresponding Quinn spectrum with singularities represents the homology of manifolds with singularities of \cite{MR0346824}. Section 5 is devoted to a Bousfield-Kan type spectral sequence for ads with singularities. For complex bordism such a spectral sequence was developed by Morava  in \cite{MR546788}. 
In section 6 we discuss product structures. It is shown that the Quinn spectrum  of an ad theory with singularities is a strict module spectrum over the original Quinn spectrum. Moreover,  there always is an external product for ad-theories with singularities which has all desired properties. Internal product structures are more difficult to obtain. We show that there is a way to produce an ad theory with an internal product which comes with a map from the original ad theory. In general, this map does not induce a homotopy equivalence on Quinn spectra in general but it does so under some conditions.
\subsubsection*{Acknowledgements.}
The authors like to thank the referee for useful suggestions.

\section{Ad theories and Quinn spectra}
\noindent In this section we recall the basic notions of \cite{LM06} which
lead to spectra of Quinn type. 
\par
Recall from \cite[Definition 3.3]{LM06} that a $\Z$-graded category $\A$ is a category
with an action of $\Z/2=\{\pm1\}$ and $\Z/2$-equivariant functors 
\[ d=dim: \A \lra \Z,\;  \emptyset : \Z \lra \A\]
which satisfy $d \, \emptyset = id$. Here, $\Z$ is regarded as a poset with
the trivial action. The full subcategory of $\A$ of $n$-dimensional objects is denoted by
$\A_n$. In abuse of notation, we will often write $\emptyset$ for the  object $\emptyset_n$ of  $\A_n$.  A $k$-morphism between graded categories is a  functor which
decreases  the dimension by $k$ and strictly commutes with the involution $-1$ and $\emptyset$.
\par
Let $K$ be a ball complex in the sense of \cite{MR0413113}. We write
${\mathcal C}ell(K)$ for the category with objects in dimension $n$ the
oriented cells $(\sigma, o)$ of $K$ and the empty cell
$\emptyset_n$. There are only identity morphisms in ${\mathcal C}ell(K)_n$
and morphisms to higher dimensional cells are given by inclusions
of cells with no requirements to the orientations. 
The category  ${\mathcal
C}ell(K)$ is a graded category  with the orientation reversing
involution. Note that morphisms between ball complexes induce
morphisms on the cellular categories. Moreover, if $L$ is a
subcomplex of $K$ we can form the quotient category ${\mathcal C}ell(K,L)
$ of ${\mathcal C}ell(K)$ by identifying the cells of $L$ with the empty cells.
\par
Next we recall the definition of an ad theory from \cite[Definition 3.8]{LM06}.
\begin{definition}\label{ad property}
Let $\A$ be a category over $\Z$. A $k$-morphism from  ${\mathcal C}ell(K,L)$ to
$\A$  is called a pre $(K,L)$-ad of
degree $k$. We write $\mbox{pre}^k(K,L)$ for the
set of these pre ads.  An {\em ad theory} is an $i$-invariant 
sub functor $\mbox{ad}^k$ of $\mbox{pre}^k$ from ball complexes to sets for each $k$
with the property $\mbox{ad}^k(K,L)= \mbox{pre}^k(K,L)\cap
\mbox{ad}^k(K)$ and
which satisfies the following axioms:
\begin{itemize}
\item[{\em (pointed)}]
the pre ad which takes every oriented cell to $\emptyset$ is an ad for
every $K$
\item[{\em (full)}]
any pre $K$-ad which is isomorphic to a $K$-ad is a $K$-ad
\item[{\em (local)}]
every pre $K$-ad which restricts to a $\sigma$-ad for each cell
$\sigma$ of $K$ is a $K$-ad.
\item[{\em (gluing)}] for each subdivision $K'$ of $K$ and each
  $K'$-ad $M$ there is a $K$-ad which agrees with $M$ on each common
  subcomplex of $K$ and $K'$.
\item[{\em (cylinder)}]
there is a natural transformation 
\[ J: \mbox{ad}^n(K) \lra \mbox{ad}^n(K\times I) \]
with the property that for every $K$-ad $M$ the restriction of $J(M)$ to $K\times 0$ and to
$K\times 1 $ coincides with $M$. It takes trivial ads to trivial ones.
\item[{\em (stable)}] let
\[ \theta: {\mathcal C}ell(K_0,L_0) \lra {\mathcal C}ell (K_1,  L_1)\]
be a $k$-isomorphism with the property that it preserves all incidence numbers
$$ [ o(\sigma),o(\sigma')]=[o(\theta \sigma),o(\theta \sigma')]$$ 
(see \cite[p.82]{MR516508}.)
Then the induced map of pre ads restricts to ads:
\[ \theta^ * : \mbox{ad}^l(K_1,L_1)\lra
\mbox{ad}^{k+l}(K_0,L_0).\] 
\end{itemize}
A multiplicative  ad theory in a graded symmetric monoidal category
${\mathcal A}$ is equipped with a natural
transformation
\[ \mbox{ad}^p (K) \wedge  \mbox{ad}^q(L) \lra  \mbox{ad}^{p+q}
(K\times L)\] and 
the object $e$ in $ \mbox{ad}^0(*)$ 
which is associative and unital in the sense of \cite[Definition 3.10 and 18.4]{LM06}.  A multiplicative ad theory is called commutative if the monoidal structure of ${\mathcal A}$ extends to a permutative structure (see \cite[Definition 3.1 and 3.3]{LM14}).
In particular, there is a natural isomorphism
$$ \gamma: x \otimes y \lra (-1)^{d(x)d(y)} y \otimes x$$
for all $x,y \in {\mathcal A}$.
\end{definition}
\begin{example}\label{R}
Let $R$ be a ring with unit. Consider $R$ as the graded  category  for which 
the objects are  the elements of $R$ concentrated in dimension 0, there are only identity
morphisms and the involution is the multiplication by -1. Then there is
a multiplicative ad theory $\mbox{ad}_R$ with $K$-ads $M$ all pre $K$-ads with the
property that for all cells $\sigma \in K$ of dimension $n$
\[ \sum_{dim(\sigma')=n-1, \sigma'\subset \sigma}
   [o(\sigma),o(\sigma')] M(\sigma',o(\sigma'))=0 \]
 where $ [o(\sigma),o(\sigma')]$  is the incidence number. If $R$ is commutative then so is $\mbox{ad}_R$.
\end{example}
\begin{example}\label{manifolds}
Let $\mathcal{ST}op$ be the graded category of compact oriented topological manifolds. An ad theory over $\mathcal{ST}op$ can be defined as follows: a
pre $K$-ad $M$ is an ad if for each $\sigma'\subset \sigma$ of one dimension lower the map
$M(\sigma', o')\lra M(\sigma ,o)$ factors through an orientation preserving  map
\[ M(\sigma', o')\lra[o,o']\partial M(\sigma, o)\]
and $\partial M(\sigma, o)$ is the colimit  of $M$ restricted to $\partial
\sigma$. See \cite[Section 6]{LM06} for details. Similarly, there is an ad theory over the graded category of compact unoriented topological manifolds  $\mathcal{T}op$. For instance,  a decomposed
(oriented) manifold
in the sense of \cite{MR0346824} is a $\Delta^ n$-ad.
\end{example}
\begin{example}
Let $W$ be the standard resolution of $\Z$ by $\Z[\Z/2]$ modules. Define the objects of ${\mathcal A}$ to be the quasi-symmetric complexes, that is, in dimension $n$ we have pairs $(C;\varphi)$ where $C$ is a quasi finite complex of free abelian groups and 
$$ \varphi : W \ra C \otimes C$$
is a $\Z/2$ equivariant map which raises the degree by $n$.  The dimension increasing morphisms $f: (C;\varphi )\ra (C';\varphi')$ are the chain maps  and for equal dimension of source and target one further assumes that 
$$ (f\otimes f) \varphi = \varphi'.$$
The involution changes the sign of $\varphi$. The $K$-ads of symmetric Poincar\'e complexes are those (balanced) functors  which
\begin{enumerate}
\item
are closed, that is, for each cell $\sigma$ of $K$ the map from the cellular chain complex 
$$\mbox{cl}(\sigma )\ra \Hom(W,C \otimes C) $$
which takes  $(\tau ,o)$ to the composite
$$
W\stackrel{\varphi_{(\tau,o)}}{\lra}C_\tau\otimes C_\tau\ra C_\sigma \otimes C_\sigma
$$
is a chain map.
\item
are well behaved, that is,  each map $f_{\tau \subset \sigma}$ and
\[
\textstyle C_{\partial \sigma}=\colim_{\tau\subsetneq\sigma} \, C_\tau\to C_\sigma
\]
are a cofibrations (split injective).
\item
non degenerate, that is, the induced map
$$
H^*(\Hom (C,\Z))
\to
H_{\dim\sigma-\deg F-*}(C_\sigma/ C_{\partial \sigma}).
$$
is an isomorphism.
\end{enumerate}
\end{example}
For an ad theory the bordism groups $\Omega^n$ are obtained by
identifying two *-ads of dimension $n$ if there is an $I$-ad which
restricts to the given ones on the ends. 

\begin{thm}[ {\cite[16.1,17.9, 18.5]{LM06} \cite[1.1]{LM14}} ] \label{Mainold}
The ads form the simplexes of the spaces in a positive
$\Omega$-spectrum $Q(ad)$ in a natural way. Its coefficients are given by the bordism
groups.  If the ad theory is multiplicative then the spectrum can be given
the structure of a symmetric ring spectrum. If the ad theory is commutative then it is weakly equivalent to a commutative ring spectrum.
\end{thm} 
In the example of a commutative ring one obtains the Eilenberg-MacLane spectrum
with $R$-coefficients. In the example of oriented manifolds one
obtains a spectrum which is homotopy equivalent to the Thom spectrum. In the example of symmetric Poincar\'e complexes the spectrum coincides with the symmetric $L$-theory spectrum. More examples can be found in \cite{BLM}.

\section{Singularities}
In this section we introduce the concept of singularities to ad theories. Bordism theories of manifolds with singularities have been studied by the first author in \cite{MR0346824}. In ordinary bordism one works with closed manifolds and bordisms between them. In the case of singularities of type $P_1$ one considers manifolds $M$ with a special boundary. (One may think that these objects are the results from removing the cones over $P_1$ from closed singular manifolds of type $P_1$.)  More specifically,
the `closed' $P_1$-manifolds are those with boundary of the form $N\times P_1$ for some closed manifold $N$. A null bordism $B$ of $M$ comes with a decomposition into two boundary components which are glued along their common boundaries. One component is $M$ and the other again comes with a homeomorphism to a product with one factor $P_1$. When studying more than one singularity at once one is forced to look at further decompositions of manifolds. As we have seen before, such objects  are provided by manifold $\Delta^n$-ads. We will reconsider the bordism theory of manifolds with singularities in more detail in Section \ref{man with sing}.   

In order to generalize this concept to other ad theories, suppose we are 
given a commutative multiplicative
ad theory $\mbox{ad}$ over $\A=(\A,\otimes ,I)$. 
\begin{definition} \label{defsing}
Let  $S=(P_1,P_2, \ldots)$ be a sequence of
  $*$-ads  and set
\[ S_n=( P_1,\ldots, P_n).\]  
Let  $\A(S_n)$
be the graded category whose objects are given by the following
data:
\begin{enumerate}
\item a pre $\sigma$-ad $M_\sigma $
for each cell $\sigma \subset \{ 0,1,\ldots ,n\}$ of
$\Delta^ n$ with the property
\[ M_\sigma = \emptyset \quad \mbox{ if } 0\not\in \sigma \]
For $\sigma=\{ 0,1,\ldots ,n\}$ we simply write $M$ for the top
pre ad.  
\item an isomorphism of pre ads for each $i\not\in \sigma$
\[  f_{\sigma, i}: \partial_i M_{(\sigma,i)} \stackrel{\cong}{\lra} (-1)^{|\sigma, i|} M_{\sigma }\otimes P_i.\]  
Here, $\partial_i$ denotes the restriction to the face $\sigma$,
$(\sigma, i)$ means $\sigma\cup \{ i\}$ and 
$$|\sigma, i|= d(P_i)\sum_{s>i, s\in \sigma}d(P_s).$$ 

We demand for each object that $\partial_0 M_\sigma=\emptyset $ and for all $i,j>0$ the diagram 
\[ \xymatrix{\partial_j \partial_i M_{(\sigma,i,j)} \ar[r]
\ar[d] ^ {=}& 
\partial_j  M_{(\sigma,j)} \otimes P_i\ar[r]
&     M_{\sigma }\otimes P_j \otimes P_i \ar[d]^ {1 \otimes \gamma} 
\\
\partial_i\partial_j M_{(\sigma,i,j)}  \ar[r]
& \partial_i    M_{(\sigma,i)}\otimes P_j   \ar[r]
&  M_{\sigma}  \otimes P_i\otimes P_j }\]
commutes (after taking the appropriate signs).
\end{enumerate}
The dimension of an object $M_\cdot$ in $\A(S_n)$ is $d(M)-n$.
Morphisms are morphisms of pre ads which commute with the isomorphisms $f_{\sigma,i}$.
\end{definition}
\begin{example}\label{3.2}
For $n=0$ an object is determined by the value of the top cell of $\Delta^0$ if $\emptyset$ is initial in $\A$. Hence we have 
$$ \A( ) = \A .$$
For $n=1$ an object is a $\Delta^1$-pre ad $M=M_{\{ 0,1\}}$ and an object $N=M_{\{0\}}$ of $\A$ such that $M$ has faces $\emptyset$ and $N\otimes P_1$.
\end{example}
\begin{lemma}\label{adjunction}
For a ball complex $L$ consider the graded category ${\mathcal B}$ of $L$-pre ads $\mbox{pre}_\A (L)$ with values in $\A$. Then there is a natural equivalence of the form
\[ \mbox{pre}_\A (K\times L) \cong \mbox{pre}_{\mathcal B}(K).\]
\end{lemma}
\begin{proof} We have a natural equivalences of categories over $\Z$
\[ {\mathcal C}ell (K)\wedge_{\mathbb Z/2} {\mathcal C}ell (L) \cong {\mathcal
  C}ell(K\times L).\]
  Here, the left hand side has objects pairs $((\sigma , o ),(\sigma' ,o'))$ of oriented cells which are identified with $\emptyset$ if one of the cells is empty. In addition, each such pair is identified with  the pair $((\sigma , -o ),(\sigma' ,-o'))$. 
The claim follows from the obvious adjunction between products and functor sets.
\end{proof}
\begin{prop}
Let $ \mbox{ad}/S_n(K)$ be the set of pre $K$-ads in $\A(S_n)$ which
give  $(K\times \sigma)$-ads 
in $\A$ under the adjunction of Lemma \ref{adjunction} for each cell
$\sigma$ of $\Delta^ n$. Then $\mbox{ad}/S_n$
defines an ad theory.
\end{prop} 
\begin{proof} The set $\mbox{ad}/S_n$ clearly is pointed and full. Suppose that
  we are given a pre $K$-ad in $\A(S_n)$ which restricts to an adjoint
  of a $\tau  \times \sigma$-ad for every $\tau \in K$ then its
  adjoint restricts to a $K\times \sigma$-ad by the locality property of Definition \ref{ad property} and hence it
  is an ad. \par Next, we check the gluing property. A subdivision
  $K'$ of $K$ defines the subdivision $K'\times \sigma$ of
  $K\times \sigma$. Hence a  $K'\times \sigma$-ad can be glued
  to a $K\times \sigma$-ad and the claim follows.\par
The cylinder $J$ of the original ad theory  takes a $K \times \sigma$-ad to a $K\times \sigma \times I $-ad and hence defines a cylinder for
  $\mbox{ad}/S_n$. \par
 Finally, we have to show the stability axiom. An incidence number preserving  $k$-isomorphism 
\[  \theta: {\mathcal C}ell(K_0,L_0) \lra {\mathcal C}ell (K_1,  L_1)\]
induces a $k$-isomorphism
\[  \theta\times id: {\mathcal C}ell(K_0\times \sigma,L_0\times \sigma ) \lra
    {\mathcal C}ell (K_1\times \sigma ,  L_1\times \sigma ).\]
Hence  for a $(K_0,L_0)$-ad$/S_n$ $M$  the ads induced by
$(\theta\times id)^ *$ of its adjoints assemble to
$\theta^ *M$.
\end{proof}
\begin{example} Let $\mbox{ad}$ be the ad theory of topological manifolds. Then the monoidal structure is the cartesian product. For $n=0$ an object of $\mbox{ad}/S_0(*)$ is  a manifold without boundary. For $n=1$ and
$S_1=(*)$ we have a manifold with an arbitrary boundary. The picture
shows a manifold with a $\Z/3$-singularity, that is, an element of
$\mbox{ad}/(P_1)(*)$ where $P_1$ consists of three points. (In other words, the local cone structure is $C\Z /3$.) In the notation of Example \ref{3.2}, the manifold $N$ consists of two points.
\begin{center}
\begin{figure}[ht]
\includegraphics[scale=0.5]{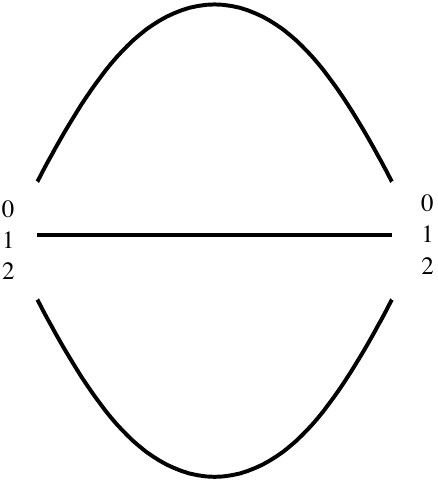}
\caption{}
\end{figure}
\end{center}
\end{example}
Next we investigate how the ad theories $\mbox{ad}/S_n$ are related
for different $n$. First observe that we have a map
\[ \mu_{P_{n+1}}: \mbox{ad}/S_n \lra \mbox{ad}/S_{n}\]
of degree $d=d (P_{n+1})$ which multiplies the ads by
$P_{n+1}$ from the right. Furthermore, there is a map of degree -1
\[ \pi: \mbox{ad}/S_{n} \lra \mbox{ad}/{S_{n+1}}\]
which comes from considering an object of $\A(S_n)$ as an object of
$\A(S_{n+1})$ by 
$$\pi M_\sigma=\left \{ \begin{array}{ll} 
M_{\sigma \setminus \{n+1\} } & \mbox{ if }  (n+1)\in \sigma\\
\emptyset &\mbox{ else}
\end{array}\right. .$$
This certainly defines a pre $K$-ad
$\pi (M)$ over $S_{n+1}$  for each $K$-ad $M$ over $S_n$.
\begin{lemma}
$\pi(M)$ is an ad.
\end{lemma}
\begin{proof} We only check the top cell $\sigma = \{0,1,\ldots,n\}$. The
  other cells are similar. The adjoint of  $M$ is a  $K\times
  \Delta^ n$-ad. The 1-isomorphism of graded categories
\[  {\mathcal C}ell(\Delta^ {n+1},
 \{ n+1\} \cup \partial_{n+1} \Delta^ {n+1})\stackrel{\cong}{\lra}  {\mathcal
   C}ell(\Delta^ n ) \]
can be multiplied with $ {\mathcal
  C}ell(K)$ and hence gives the desired ad by the stability  axiom.
\end{proof}
Finally, we have a map 
\[ \delta : \mbox{ad}/S_{n+1}\lra \mbox{ad}/S_n.\]
It takes a $K$-ad $M$ over $S_{n+1}$ to the $K$-ad over $S_n$ given by the formula
\[ \delta (M) (\sigma, o) =  M(\sigma,o)_{|\{ 0,1,\ldots ,n\}}.\]  
\begin{thm}\label{exact}
Let $\Omega^ {S_n}_*$ be the bordism group of the ad theory
$\mbox{ad}/S_n$. Then the sequence
\[ \ldots \stackrel{\delta_*}{\lra} \Omega^ {S_n}_*\stackrel{{\mu_{P_{n+1}}}_*}{\lra} \Omega^
   {S_n}_{*+d}\stackrel{\pi_*}{\lra }\Omega^ {S_{n+1}}_{*+d-1}\stackrel{\delta_*}{\lra}
    \Omega^ {S_n}_{*-1}\stackrel{{\mu_{P_{n+1}}}_*}{\lra}
   \ldots 
\]
is exact. 
\end{thm}
\begin{proof}
The proof is essentially the same as in  \cite[Theorem 3.2]{MR0346824}. 
\end{proof}
\begin{example}
Consider the sequence $S=(\emptyset , \emptyset,\ldots )$. Using the
suspension axiom it is not hard to see that 
$\mbox{ad}/S_n$ consists of $n+1$ copies of the original Quinn spectrum. Hence the above exact sequence consists of short
split exact sequences.    
\end{example}
\begin{example}
Let $R$ be a ring and suppose
$x\in R$ is a non zero divisor. Consider the ad theory of Example
\ref{R}.
Then the maps of spectra induced by $\mu_x, \pi$ and $\delta$
corresponds to the Bockstein exact sequence in singular homology.
\end{example}

\section{Example: Manifolds with singularities and assemblies}\label{man with sing}
In this section we look at the ad theory of compact manifolds. For
simplicity we restrict our attention to the unoriented topological
case. It will then be clear how to do other cases of bordism theories.\par
We fix a sequence $S$ of closed manifolds and write $Q[X]$ for the
Quinn spectrum of $\mbox{ad}[X]/S_n$. Here, $X$ is a topological space
and $\mbox{ad}[X]$ is defined as in example \ref{manifolds} with
singular manifolds in $X$, that is, manifolds equipped with a continuous map to $X$. In the following, we call a simplicial set without the data of degeneracies a `semi simplicial set' (another  name in the literature is `$\Delta$-set'.)\par
\begin{prop}
Suppose $F$ is a functor from semi simplicial sets to the category of symmetric spectra which sends homotopy equivalences to stable equivalences. Then there
is a natural transformation in the homotopy category 
\[  F[*]\wedge |X|_+ \lra F [X]\]
which is the obvious equivalence if  $X$ is a point.
\end{prop}
\begin{proof} Some versions of the desired transformation are certainly known and run under the name assembly map (see for example the discussion in \cite{MR1388318}). In this simple form it can be obtained as follows:
for a
semi simplicial set
$X$ we have the natural homotopy equivalence
\[   \mbox{hoco}\! \lim_{\!\!\!\!\!\!\!\!\!\!\!\! \Delta^ n\ra
  X}F[\Delta^ n ]\lra  \mbox{hoco}\! \lim_{\!\!\!\!\!\!\!\!\!\!\!\!
  \Delta^ n\ra X}F[*]\cong  \mbox{co}\! \! \lim_{\!\!\!\!\!\!\!
  \Delta^ n\ra X}F[*]\wedge |\Delta^ n|_+ \cong  F[*]\wedge |X|_+\]
whose homotopy inverse can be composed with the map 
\[  \mbox{hoco}\!  \lim_{\!\!\!\!\!\!\!\!\!\!\!\! \! \! \sigma: \Delta^ n\ra
  X}F[\Delta^ n]
\lra   \mbox{co}\!\! \!  \! \lim_{\!\!\!\!\!\!\!
 \sigma:  \Delta^ n\ra X}F[\Delta^ n]\stackrel{(F[\sigma])}{\lra}
F[X].
\] 
\end{proof}
\begin{thm}
The spectrum $Q=Q[*]$ represents the homology theory of manifolds with
singularities of \cite{MR0346824}.
\end{thm}

\begin{proof}
We first show that the bordism groups of  $\mbox{ad}/S_n$ are
naturally equivalent to the bordism groups of manifolds with
singularities $S_n$. For that, recall that a *-ad in
$\mbox{ad}/S_n$ consists of $\sigma$-ads $M_\sigma$ with
$M_\sigma=\emptyset$ if $0\not\in \sigma$ and a system of compatible
isomorphisms
\[ \partial _i M_{(\sigma ,i)} \cong (-1)^{|\sigma ,i|} M_\sigma \times P_i.\] 
Hence, it defines a closed $S_n$-manifold and a closed
$S_n$-manifold gives a $*$-ad.
A null bordism $B$ is a family of $I\times
\sigma$-ads with one end empty and the other is the bounding
object $M$.
The situation is illustrated for $n=1$ in Figure \ref{half egg}.
\begin{center}
\begin{figure}[ht]
\includegraphics[scale=0.5]{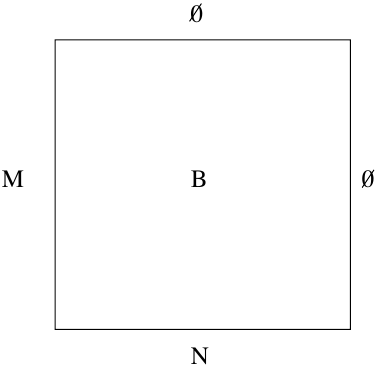}
\caption{}\label{half egg}
\end{figure} 
\end{center}
Here, $B$ and $M$ are objects of ${\mathcal T}op(P_1)$ and there are homeomorphisms
\begin{eqnarray*}
\partial_1B&\cong& N\times P_1\\
\partial M&\cong & \partial N\, \times P_1\\
\partial B&\cong & M\cup_{\partial N\, \times P_1}N\times P_1.
\end{eqnarray*}  
For example in the case of  $\Z/3$-manifold $M$ considered
earlier, a null bordism is pictured in Figure \ref{3null}. Here, $N$ is a horizontal arc.
\begin{center}
\begin{figure}[ht]
\includegraphics[scale=0.28]{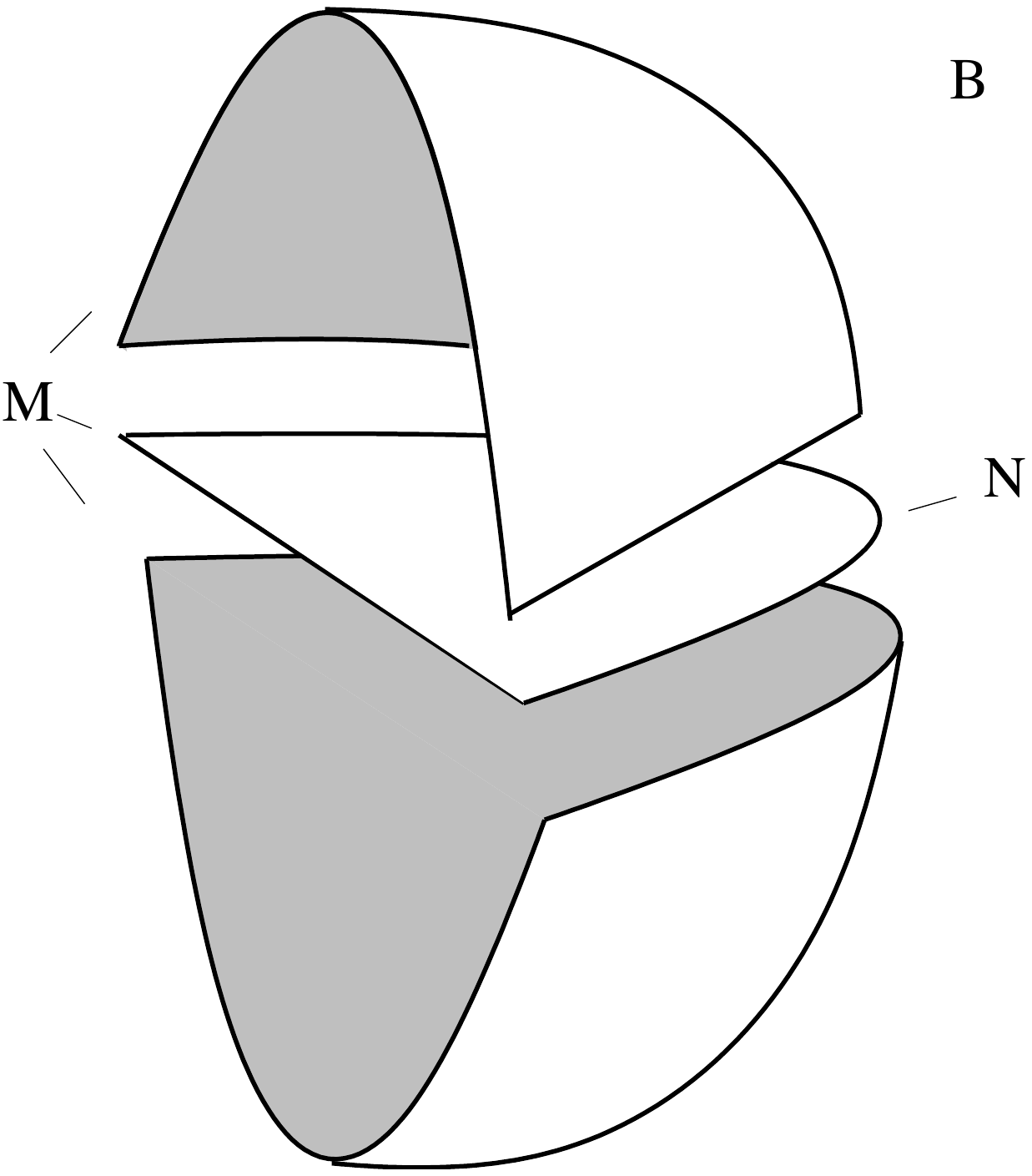}
\caption{}\label{3null}
\end{figure}
\end{center}
This
shows that the bordism groups coincide. The same argumentation shows that the bordism groups  of the ad theory  $\mbox{ad}[X]/S_n$ coincide with the bordism groups of manifolds with singularities in $X$.
This implies that $Q$ is a
homotopy invariant functor and hence the assembly map is well defined by the preceding proposition. \par
Finally, we have to show that the assembly map is a homotopy
equivalence. Using the fact that bordism of singular $S_n$-manifolds
defines a homology theory we know that the functor
\[ (X,Y) \mapsto \pi_*(Q[X],Q[Y])\] together with the boundary
operator defines a homology theory as well. Thus the assembly map
defines a natural transformation between homology theories and is an
isomorphism for a point. Thus the claim follows from the comparison
theorem between homology theories. 
\end{proof}

\section{A Bousfield-Kan spectral sequence}
The exact sequences of the $\mbox{ad}/S_n$-bordism groups for different $n$
are part of a spectral sequence of Bousfield-Kan type. In the case of
classical complex bordism it first has been developed in
\cite{MR546788}.
\par
Fix a commutative multiplicative ad theory $\mbox{ad}$ over $\mathcal{A}$ and a sequence of *-ads $S=(P_1,P_2,  \ldots, )$. We will assume that $\emptyset$ is initial in $\mathcal{A}$. For a finite set $T=\{ t_1,t_2,\ldots, t_n\}$ of natural numbers with $t_i\leq t_{i+1}$ for all $i$
let $$S_T=(P_{t_1},\ldots ,P_{t_n})$$ be the subsequence of $S$ indexed by $T$.

For each $T$ we are going to define the graded category $\A \left< S_T\right>$. The notation is taken from the theory of manifolds with faces \cite[Section 2.1]{MR1781277} and should not be confused with $\A (S_T)$.  The objects of $\A \left< S_T\right>$ consists of the following data:
\begin{enumerate}
\item a  pre $*$-ad $M_\sigma $
for each cell $\sigma \subset \{ 0,1,\ldots ,n\}$ of
$\Delta^ n$ with the property
\[ M_\sigma = \emptyset \quad \mbox{ if } 0\not\in \sigma \]

\item isomorphisms for each $i\not\in \sigma$
\[  f_{\sigma, i}:  M_{(\sigma,i)} \stackrel{\cong}{\lra}  (-1)^{|\sigma ,i|} M_{\sigma }\otimes P_{t_i} \]  
which are compatible with the face maps.
Moreover,  $\partial_0 M=\emptyset $ and for all $i,j>0$ the diagram 
\[ \xymatrix{M_{(\sigma,i,j)} \ar[r]
\ar[d] ^ {=}& 
 M_{(\sigma,j)} \otimes P_{t_i}\ar[r]
&     M_{\sigma }\otimes P_{t_j} \otimes P_{t_i} \ar[d]^ {1 \otimes \gamma} 
\\
M_{(\sigma,i,j)}  \ar[r]
&    M_{(\sigma,i)}\otimes P_{t_j}   \ar[r]
&  M_{\sigma}  \otimes P_{t_i}\otimes P_{t_j} }\]
commutes (with the appropriate signs).
\end{enumerate}

\begin{example}
The category $\A \left< \{ 1 \} \right>$ consists of $*$-ads $M_\sigma$ for $\sigma \in \Delta^1$ together with isomorphisms 
$$ M_{\{ 0,1\}} \cong M_{\{ 0\} }\otimes P_1$$
and $M_{ \{ 1 \} } =M_\emptyset =\emptyset$. An object of $\A \left< \{ 1,2\}\right>$ is a collection of $*$-ads $M_\sigma$ for $\sigma \in \Delta^2$ together with isomorphisms
\begin{eqnarray*}
 M_{\{ 0,1,2\} }&\cong & (-1)^{d(P_1)d(P_2)}M_{\{ 0,2\} } \otimes P_{1}\\
 M_{\{ 0,1,2\} }&\cong  & M_{\{ 0,1\} } \otimes P_{2}\\
 M_{\{ 0,1\} }&\cong &  M_{\{ 0\} } \otimes P_{1}\\
 M_{\{ 0,2\} }&\cong  & M_{\{ 0\} } \otimes P_{2}\\
  M_\sigma&=&\emptyset \qquad\mbox{ if } 0\not\in \sigma
 \end{eqnarray*}
which let the diagram above commute. In particular, we have a distinguished isomorphism
$$  M \cong  M_{\{ 0\} } \otimes P_{1}\otimes P_2.$$
\end{example}
Let $\delta^ k S_T$ be
the sequence obtained from $S_T$ with the $k$th entry omitted. 
 Consider the cubical diagram of graded categories
with vertices $\A\left< S_T \right> $ and face functors 
\[ \slashed{\partial} _k:  \A\left< S_T\right>  \lra  \A\left< \delta^ k S_T \right> \]
given by 
\[\slashed{\partial}_k (M_\cdot)_\sigma  =(-1)^{|\sigma ,k|^-} M_{(\sigma,k)}\]
with
$$ |\sigma ,k|^-=d(P_k)\sum_{0<s<k,s\in \sigma}d(P_s)$$
and  isomorphisms $(-1)^{|\sigma ,k|^-}f_{((\sigma,k),i)}$.
\begin{example}
Explicitly, the boundary functor 
$$ \slashed{\partial}_2:  \A \left< \{ 1,2 \} \right> \lra \A \left< \{ 1 \} \right>$$
is given by
$$ (\slashed{\partial}_2M)_{\{ 0,1\} } = (-1)^{d(P_1)d(P_2)} M_{\{ 0,1,2\} } \cong M_{ \{ 0,2\} } \otimes P_1 = 
(\slashed{\partial}_2M)_{\{ 0 \} }\otimes P_1.$$
\end{example}
For each set $T$ there is an ad theory $\mbox{ad}\left<  S_T \right> $ over $\A \left< S_T\right>$ as follows: the $K$-ads are those pre ads $M$ over $\A \left< S_T\right>$ which are cell-wise 
$K$-ads in $\A$, that is, for each cell $\sigma$ of $\Delta^n$ the functor $M_\sigma$ from ${\mathcal C}ell(K)$ to $\A$ is a $K$-ad. 

\begin{lemma}
\begin{enumerate}
\item
For each vertex $T$ there is a canonical  isomorphism
\[ \pi_* Q (ad\left< S_T \right>)  \cong \pi_* Q (ad).\]
\item
The  map
\[{\slashed{\partial}_k}_* :\pi_* Q (ad\left< S_T \right> )\lra  \pi_* Q (ad \left< {\delta^ k S_T}\right> )\]
induced by the face map is given under the above isomorphism by
multiplication by $P_{t_k}$ map on the bordism group $\Omega_*$.
\end{enumerate}
\end{lemma}
\begin{proof}
The homotopy equivalence is induced by the functors
$$
\phi: \A \left< S_T\right>  \lra \A;  \qquad M  \mapsto   M_{ \{ 0\} } 
$$ and 
$$\psi:  \A   \lra  \A\left< S_T\right>;  \qquad  M \mapsto  
 (M_{ \sigma})_{\sigma \subset \Delta^n}
$$ 
with $ M_\sigma=M\bigotimes_{i\in \sigma \setminus \{ 0\}} P_{t_i}$. The composite $\phi \psi$ is the identity and  $\psi \phi(M)$ is canonically isomorphic to $M$.  By construction, these functors induce maps of ad theories and hence maps between Quinn spectra. Since isomorphic objects are bordant the claim follows. For the second claim observe that the composite
$$ \A   \lra  \A\left< S_T\right>\stackrel{{\slashed{\partial}_k}_*}{\lra}   \A\left< \delta^k S_T\right> \lra \A$$
sends $M$ to $M\otimes P_{t_k}$ by definition.
\end{proof}
The ad theories $\mbox{ad}\left<  S_T \right> $ and the face functors (with the appropriate signs) define a cubical diagram of ad theories which will be denoted by $\mbox{ad}\left<  S \right> $ in the sequel.

\begin{lemma}
Let $Q(ad \left<  S_n\right> )^ +$ be the $n+1$-dimensional diagram  indexed by
subsets $T$ of $\{0,1,\ldots, n\}$ without $\emptyset$ given by
\[ Q(ad \left<  S_n\right> )^ + (T) = Q(ad \left<  S_T\right> )\]
if $ 0\not\in
  T$ and by * else.
Then we have 
\[ \hocolim Q(ad \left<  S_n\right> )^ + \simeq  Q(ad/S_n).\]
\end{lemma}
 \begin{proof}
 The proof is an induction on the number of singularities. In the case of only one singularity the left hand side is the  cokernel of the map of Quinn spectra induced by $\slashed{\partial}_1$. This map is given by the multiplication by $P_1$.  On the other hand the map which considers an ad over $\A\left< \emptyset \right> = \A$ as an object of $\A (S_1)$ can be composed with the multiplication by $P_1$ map. The cokernel of the induced map of spectra  has the same homotopy type as $Q(ad/S)$ by the exact sequence of Theorem \ref{exact}.  Hence we get a map between the cokernels which induces an isomorphism on bordism groups. \par
The inductive step follows since we have canonical equivalences
\begin{eqnarray*}
(\mathcal{A} (S_n))(P_{n+1})& \cong & \mathcal{A}(S_{n+1}) \\
(\mathcal{A} \left< S_n\right>)\left< P_{n+1}\right> & \cong & \mathcal{A}\left< S_{n+1}\right>.
\end{eqnarray*}
(Note that the multiplication with $*$-ads is still defined (compare \ref{module} below) even though the categories may not be monoidal.) 
 \end{proof}
 The homotopy colimit identification of the spectra with singularities
 furnishes a spectral sequence of Bousfield-Kan type \cite{MR0365573}  with $E_2$-term the homology of the chain complex
$$ \xymatrix{ \ldots \ar[r]& \bigoplus_{\# T=k} \pi_* Q( ad \left< S_T
   \right >)\ar[r]^{\slashed{\partial} }&  \bigoplus_{\# T=k-1} \pi_* Q( ad \left< S_T \right >)\ar[r]& \ldots .}
  $$ 
with $\slashed{\partial} = \sum (-1)^k\slashed{\partial}_k$. This gives
 \begin{thm}
 There is a spectral sequence converging to the bordism groups of $\mbox{ad}/S_n$ with $E_2$-term  the homology of the  Koszul complex $K(P_1,\ldots ,P_n)$, that is, the tensor product over $\Omega_* $ of the complexes
 $$ \xymatrix{0 \ar[r] &\Omega_* \ar[r]^{\cdot P_k}\ar[r]&\Omega_* \ar[r]^{}\ar[r]&0}.$$
 \end{thm}
 
 \section{Product structures}
In this section we will investigate product structures on ad theories with singularities.
We ask which multiplicative structures are
inherited from an ad theory to its $S_n$-ad theory $\mbox{ad}/S_n$. 

Suppose the ad theory is multiplicative. Then  clearly 
we have an action 
\[ \mbox{ad}(K) \times (\mbox{ad}(L)/S_n) \lra \mbox{ad}(K\times L)/S_n \]
and hence, we obtain a module structure on the Quinn spectra. 
\begin{cor}\label{module}
The Quinn spectrum of the ad theory with
singularities is a strict module spectrum
over the original Quinn spectrum.
\end{cor}
Further product structures come from the following external product:
we start with a multiplicative ad theory and finite sequences $P=(P_1,P_2, \ldots, ,P_n)$ and $Q=(Q_1,Q_2,\ldots,Q_m)$. Write $(P,Q)$ for the  sequence
$$ (P_1,P_2, \ldots, ,P_n,Q_1,Q_2,\ldots,Q_m).$$
\begin{prop}
There is an external product
$$\times: ad/P(K) \times ad/Q(L) \lra ad/(P,Q)(K\times L)$$
which is natural and associative.
\end{prop}
\begin{proof}
Suppose $M$ is a $K$-ad$/(P)$ and $N$ is an $L$-ad$/(Q)$. For  a subset $\rho$ of $$\{ 0,1,\ldots ,n+m\}$$ set
$$\rho_0=\rho \cap \{0,1,\ldots n\}$$
and 
$$\rho_1=((\rho \cap \{  n+1, n+2, \ldots ,n+m\} )-n)\cup (\rho \cap \{ 0\}).$$
Then the external product is given by
$$  (M\times N)_\rho = (-1)^{\sum_{s\in \rho_0,t\in \rho_1}d(P_s)d(P_t)}M_{\rho_0 } \times N_{\rho_1}.$$
The claimed properties are readily verified.
\end{proof}
Internal product structures are much harder to construct. In \cite{MR546788} an internal product on the level of homotopy groups was obtained with the help of a retraction map which reduces the singularity of type $(P,P)$ to $P$ under suitable hypothesis. In order to rigidify the products we will proceed differently. Instead of looking for a retraction map we construct a new ad theory which comes with an internal  product and is homotopy equivalent to the old one in good cases.
\begin{definition}
Let $\tau \in \Sigma_n$ be a permutation and $P=(P_1,P_2,\ldots ,P_n)$ be a sequence. Let $\tau P$ be the sequence $(P_{\tau_1},\ldots,P_{\tau_n})$ and let $$
\tau^* : \A (P) \lra \A(\tau P)$$
be the map which sends an object $M_\cdot$ to the object 
$$ (\tau^*M_\cdot)_\sigma= (-1)^{|\tau,\sigma|} M_{\tau \sigma}.$$
Here, $|\tau ,\sigma |$ is the determinant of the matrix which is obtained by permuting the $d(P_i)$-dimensional blocks of a $\sum_{i\in \sigma} d(P_i)$-dimensional identity matrix  through $\tau$.  
\end{definition}
\begin{lemma} The functor $\tau^*$ induces an isomorphism of ad theories
$$\tau^* : ad/P \lra ad/(\tau P).$$
\end{lemma}
\begin{proof}The proof is a consequence of the stability axiom.
\end{proof}
\par
Consider the self map
$$ 1+(-1)^n\tau_{n,n}: ad/(P,P) \lra ad/(P,P) $$
where $\tau_{n,n}$ is the permutation which twists the two blocks of length $n$. An ad in the image of this map has the property that for each oriented cell of the ball complex we have a $\Delta^{2n}$-ad whose $k$th face for $k\leq n$ coincides with the $k+n$ th face after permuting the two summands and applying $(-1)^n$. Moreover, the object on the top cell is twice the original object.
   
\begin{definition}
Let $P$ be an arbitrary  sequence. An $K$-ad of $\mbox{ad}/(P,P)$ is said to be {\it close to} a $K$-$\mbox{ad}/P$ if  
for each oriented cell of $K$ the value $M$ satisfies:
$$\partial_k M = (-1)^n \partial_{k+n} M\mbox{ for all } 0\leq k\leq n$$
and the same holds for the maps induced by the inclusions into the top
cell. In other words, $M$ is fixed under the action of $(-1)^ n
\tau_{n,n}$. We write $cl(ad/P)$ for all ads in $\mbox{ad}/(P,P)$ which are
isomorphic to ones which are close to $\mbox{ad}/P$.  We say that $\mbox{ad}/P$ is {\it well behaved} if each of the theories $cl(ad/(P), cl(ad/(P,P)), \ldots$ is an ad theory.
\end{definition}

\begin{definition}
Let $\pi: ad/P\lra ad/(P,P)$ be the inclusion map considered in section 3. It comes from the map $\A(P)\lra \A(P,P)$ which fills  $\emptyset$ in the faces which do not contain the last $n$ indices. Set
$$ \rho_P = (1+(-1)^n \tau_{n,n} )\pi$$ and let $\mbox{ad}/\!/P$ be the colimit of the sequence
$$ \xymatrix{ad/P \ar[rr]^{\rho_P} && cl(ad/P)\ar[rr]^{\rho_{(P,P)} }&& cl(ad/(P,P))\ar[rr]^{\; \; \ \; \; \rho_{((P,P),(P,P))} }&&\ldots}$$

\end{definition}

\begin{thm}\label{einfty}
Let $\mbox{ad}/P$ be well behaved. Then we have 
\begin{enumerate}
\item
$\mbox{ad}/\!/P$ is a commutative multiplicative ad theory. 
\item
the canonical map from $\mbox{ad}$ to  $\mbox{ad}/\!/P$ respects the multiplication.
\item
the canonical map from the spectrum $Q(ad/P)$ to $Q(ad/\!/P)$ is a
homotopy equivalence if 2 is inverted, $P$ is regular and the cylinder of $P$ admits an involution reversing isomorphism.
\end{enumerate}
\end{thm}
\begin{proof}
The product of two ads $M,N$ of the colimit, say in $cl(ad/(P,\ldots ,P))$ is given by their 
symmetrized exterior product  $(1+(-1)^n \tau_{n,n} )(M\times N)$. This definition is independent of $n$ by the hypothesis on $\A$.
Clearly, the product is compatible with the map from $\mbox{ad}$.
\par
The last assertion is more involved.  It relies on arguments which are similar to the ones given in  \cite{MR546788} for complex bordism. For simplicity we look at the case of only one singularity $P$ of dimension $m$. We have short exact sequences
$$  \xymatrix{ 0 \ar[r] & \Omega_{*-m} \ar[r]^P & \Omega _* \ar[r]& \Omega^P_* \ar[r] & 0\\ 0 \ar[r] & \Omega_{*}^P \ar[r] & \Omega _*^{(P,P)} \ar[r]& \Omega^P_{*-m-1} \ar[r] & 0.}$$
In particular, $ \Omega _*^{(P,P)}$ is a free $\Omega^P_*\cong
(\Omega/P)_*$-module on the generator 1 and a generator $\delta$ of dimension
$m+1$. (We used here the fact that the obstruction for the vanishing
of the 
multiplication by $P$ map in a theory with singularities which contain
$P$ can be
described by the bordism class of the mapping torus of $P$, see
\cite{MR0377856} for the classical case).\par
  A convenient choice of
$\delta $ is provided by the the suspension of $P$, that is the
cylinder of $\delta$ on the top cell and with $P$ as first and second
face. It maps to the unit of $\Omega_*^ P$. Since the cylinder of $P$
admits an involution reversing isomorphism we see that $-\tau_{1,1}
\delta$  is isomorphic to $i \delta$. Hence, $1-\tau_{1,1}$
annihilates $\delta$ in the bordism group. 
\par
Hence, the map
$$ Q ( ad/P) \lra Q (cl(ad/P))$$
is a weak equivalence. An inverse on the level of homotopy groups is given by
$$ \pi_* Q ( cl (ad/P)) \lra \pi_* Q (  ad/(P,P)) \lra \pi_* Q( ad/P),$$
the last map being induced by  $(1-\tau_{1,1})/2$.
\par
The same method applies to the other maps of the colimit. Note that the
obstructions for the vanishing of the multiplication by $P$-map vanish
and hence we get short exact sequences and can proceed as before. The
general case for arbitrary many singularities is analogues.
\end{proof}
\begin{cor}
Under the conditions of Theorem \ref{einfty}(iii) the Quinn spectrum $Q(\mbox{ad}/P)$ is homotopy equivalent to a commutative ring spectrum.
\end{cor}
\begin{proof}
This immediately follows from Theorem \ref{einfty} and Theorem \ref{Mainold}.
\end{proof}
 \bibliographystyle{amsalpha}
\bibliography{bordmult} 

\providecommand{\bysame}{\leavevmode\hbox to3em{\hrulefill}\thinspace}
\providecommand{\MR}{\relax\ifhmode\unskip\space\fi MR }
\providecommand{\MRhref}[2]{%
  \href{http://www.ams.org/mathscinet-getitem?mr=#1}{#2}
}
\providecommand{\href}[2]{#2}
\begin{thebibliography}{WW95}

\bibitem[Baa73]{MR0346824}
Nils~Andreas Baas, \emph{On bordism theory of manifolds with singularities},
  Math. Scand. \textbf{33} (1973), 279--302 (1974). \MR{0346824 (49 \#11547b)}

\bibitem[BK72]{MR0365573}
A.~K. Bousfield and D.~M. Kan, \emph{Homotopy limits, completions and
  localizations}, Springer-Verlag, Berlin, 1972, Lecture Notes in Mathematics,
  Vol. 304. \MR{0365573 (51 \#1825)}

\bibitem[BLM]{BLM}
Markus Banagl, Gerd Laures, and James~E. McClure, \emph{The {L}-homology
  fundamental class for {IP}-spaces and the stratified {N}ovikov conjecture},
  preprint, arXiv:1404.5395, 2013.

\bibitem[BRS76]{MR0413113}
S.~Buoncristiano, C.~P. Rourke, and B.~J. Sanderson, \emph{A geometric approach
  to homology theory}, Cambridge University Press, Cambridge, 1976, London
  Mathematical Society Lecture Note Series, No. 18. \MR{0413113 (54 \#1234)}

\bibitem[JW75]{MR0377856}
David~Copeland Johnson and W.~Stephen Wilson, \emph{{$BP$} operations and
  {M}orava's extraordinary {$K$}-theories}, Math. Z. \textbf{144} (1975),
  no.~1, 55--75. \MR{0377856 (51 \#14025)}

\bibitem[Lau00]{MR1781277}
Gerd Laures, \emph{On cobordism of manifolds with corners}, Trans. Amer. Math.
  Soc. \textbf{352} (2000), no.~12, 5667--5688 (electronic). \MR{1781277
  (2001i:55007)}

\bibitem[LM]{LM14}
Gerd Laures and James~E. McClure, \emph{Commutativity properties of {Q}uinn
  spectra}, preprint, arXiv:1304.4759, 2013.

\bibitem[LM14]{LM06}
\bysame, \emph{Multiplicative properties of {Q}uinn spectra}, Forum Math.
  \textbf{26} (2014), no.~4, 1117--1185. \MR{3228927}

\bibitem[Mor79]{MR546788}
Jack Morava, \emph{A product for the odd-primary bordism of manifolds with
  singularities}, Topology \textbf{18} (1979), no.~3, 177--186. \MR{546788
  (80k:57063)}

\bibitem[Sul67]{S}
Dennis Sullivan, \emph{Geometric {T}opology {S}eminar {N}otes}, Princeton
  University, Princeton, 1967.

\bibitem[Whi78]{MR516508}
George~W. Whitehead, \emph{Elements of homotopy theory}, Graduate Texts in
  Mathematics, vol.~61, Springer-Verlag, New York, 1978. \MR{516508
  (80b:55001)}

\bibitem[WW95]{MR1388318}
Michael Weiss and Bruce Williams, \emph{Assembly}, Novikov conjectures, index
  theorems and rigidity, {V}ol.\ 2 ({O}berwolfach, 1993), London Math. Soc.
  Lecture Note Ser., vol. 227, Cambridge Univ. Press, Cambridge, 1995,
  pp.~332--352. \MR{1388318 (97f:55005)}

\end{thebibliography}
\end{document}